\documentclass[11pt,letterpaper]{amsart}

\usepackage{xcolor}

\makeatletter
\@namedef{subjclassname@2010}{%
  \textup{2020} Mathematics Subject Classification}
\makeatother

\theoremstyle{plain}
             \newtheorem*{theorem}{Theorem}

            \newtheorem{lemma}{Lemma}
          \newtheorem{definition}{Definition}
            \newtheorem{example}{Example}

\begin{document}

\title{Topological stability from a measurable viewpoint}

\author{Keonhee Lee}
\address{Department of Mathematics, Chungnam Ntional University, Daejeon 305-764, Republic of Korea.}
\email{khlee@cnu.ac.kr}

\author{Seunghee Lee}
\address{Department of Healthcare Data Science Center, Konyang University Hospital, Daejeon, Republic of  Korea }
\email{shleedynamics@gmail.com}

\author{C.A. Morales}
\address{Instituto de Matem\'atica, Universidade Federal do Rio de Janeiro, P. O.
Box 68530, 21945-970 Rio de Janeiro, Brazil.}
\email{morales@impa.br}

\thanks{Work supported by Basic Science Research Program through the NRF funded by the Ministry of Education (Grant Number: 2022R1l1A3053628).}

\keywords{Borel probability measure, Expansive, Weak shadowing property}

\subjclass[2010]{Primary  37B25, Secondary 37C50}





\begin{abstract}
We introduce the {\em $\mu$-topological stability}. This is a type of stability depending on the measure $\mu$ different from the set-valued approach \cite{lm}.
We prove that the map $f$ is $m_p$-topologically stable
if and only if $p$ is a topologically stable point ($m_p$ is the Dirac measure supported on $p$).
On closed manifolds of dimension $\geq2$ we prove that
every $\mu$-topologically stable map has the $\mu$-shadowing property for finitely supported measures $\mu$.
Moreover the $\mu$-topological stability is invariant under topological conjugacy or restriction to compact invariant sets of full measure.
We also prove for expansive maps that the set of measures $\mu$ for which the map is $\mu$-topologically stable is convex.
We analyze the relationship between $\mu$-topological stability for absolutely continuous measures.
In the nonatomic case we show that the $\mu$-topological stability implies the set-valued stability approach in \cite{lm}.
Finally, we show that every expansive map with the weak $\mu$-shadowing property (c.f. \cite{lr}) is $\mu$-topologically stable.
\end{abstract}

\maketitle



\section{Introduction}

\noindent
The concept of {\em topological stability} by Walters \cite{b,w} is the topological version of Andromov-Pontriagin's structural stability \cite{ap}. Hyperbolic systems were proved to be topologically stable.
These include the Anosov, Morse-Smale or Axiom A diffeomorphisms with the strong transversality condition \cite{n1}, \cite{w0}.
The topological stability of Anosov diffeomorphisms was generalized through Walters's stability theorem \cite{w}.
This theorem asserts that every expansive homeomorphism with the shadowing property of a compact metric space is topologically stable.

In this paper we introduce the {\em $\mu$-topological stability}. This is a type of stability depending on the measure $\mu$ different from the set-valued approach \cite{lm}.
We prove that the map $f$ is $m_p$-topologically stable
if and only if $p$ is a topologically stable point.
On closed manifolds of dimension $\geq2$ we prove that
every $\mu$-topologically stable map has the $\mu$-shadowing property for finitely supported measures $\mu$.
Moreover the $\mu$-topological stability is invariant under topological conjugacy or restriction to compact invariant sets of full measure.
We also prove for expansive maps that the set of measures $\mu$ for which the map is $\mu$-topologically stable is convex.
We analyze the relationship between $\mu$-topological stability for absolutely continuous measures.
In the nonatomic case we show that the $\mu$-topological stability implies the set-valued stability approach in \cite{lm}.
Finally, we show that every expansive map with the weak $\mu$-shadowing property (c.f. \cite{lr}) is $\mu$-topologically stable.

\subsection{Background and motivations}
Let $X$ be a metric space. Denote by $C^0(X)$ the set of continuous maps from $X$ to itself.
Define the $C^0$-distance between maps of metric spaces $l,r:Y\to X$ by
$$
d_{C^0}(l,r)=\sup_{y\in Y}d(l(y),r(y)).
$$
If $Y\subset X$, then we denote the inclusion by $i_Y: Y\hookrightarrow X$.
In particular, $i_X$ is the identity of $X$.

We recall the basic definition of topological stability by Walters \cite{w}.

\begin{definition}
We say that $f:X\to X$
is {\em topologically stable} if for every $\epsilon>0$ there is $\delta>0$ such that for every continuous $g:X\to X$ with $d_{C^0}(f,g)\leq\delta$ there is $h:X\to X$ continuous such that
$$
d_{C^0}(h,i_X)\leq\epsilon\quad\mbox{ and }\quad f\circ h=h\circ g.
$$
\end{definition}

{\em Walters's stability theorem} \cite{w} asserts that the following two properties suffice for a continuous map of a compact metric space $f:X\to X$ to be topologically stable:

\begin{itemize}
\item
$f$ is {\em expansive} if there is $e>0$ (called {\em expansivity constant}) such that if $x,y\in X$ and $d(f^n(x),f^n(y))\leq e$ for every nonnegative integer $n$, then $x=y$.

\item
$f$ has the {\em shadowing property} if for every $\epsilon>0$ there is $\delta>0$ such that for every sequence $(x_n)_{n\geq0}$ with $d(f(x_n),x_{n+1})\leq\delta$ for all $n\geq0$ there is $x\in X$ such that
$d(f^n(x),x_n)\leq \epsilon$ for every $n\geq 0$.
\end{itemize}

A measurable version of this theorem was given in \cite{lm} or \cite{lr}.
Let us remind it.

A {\em Borel probability measure} is a $\sigma$-additive measure $\mu$ defined in the Borel $\sigma$-algebra $\beta_X$ of $X$ such that $\mu(X)=1$. Denote by
$\mathcal{M}(X)$ the set of Borel probability measures of $X$.

Denote $2^X$ the set of subsets of a compact metric space $X$.
Maps $H:X\to 2^X$ are called {\em set-valued maps}.
The domain of $H$ is defined by $Dom(H)=\{x\in X:H(x)\neq\emptyset\}$.
We say that $H$ is {\em compact-valued} if $H(x)$ is compact for every $x\in X$.
We write $d_{C^0}(H,i_X)\leq \epsilon$ for some $\epsilon>$ if $H(x)\subset B[x,\epsilon]$ where $B[x,\epsilon]$ denotes the closed $\epsilon$-ball around $x$.
We say that $H$ is {\em upper semicontinuous} if for every $x\in Dom(H)$ and every neighborhood $O$ of $x$ there exists $\eta>0$ such that
$H(y)\subset O$ for every $y\in X$ with $d(x,y)\leq \eta$.

\begin{definition}
Given $\mu\in\mathcal{M}(X)$ we call $f\in C^0(X)$
{\em $\mu$-topologically set-valued stable} if for every $\epsilon>0$ there is $\delta>0$ such that for every $g\in C^0(X)$ with $d_{C^0}(f,g)\leq\delta$ there is
an upper semicontinuous compact-valued map $H$ of $X$ such that
$$
\mu(Dom(H))\geq 1-\epsilon,\quad \mu\circ H=0,\quad d_{C^0}(H,id_X)\leq \epsilon, \quad f\circ H=H\circ g.
$$
\end{definition}

Sufficient conditions for $\mu$-topologically set-valued stable can be obtained from the next definitions where $f\in C^0(X)$ and $\mu\in\mathcal{M}(X)$:

\begin{itemize}
\item
We say that $f$ is {\em $\mu$-expansive} if there is $\epsilon>0$
such that
$$
\mu\{y\in X:d(f^n(x),f^n(y))\leq\epsilon,\,\forall n\geq0\})=0,\quad\quad\forall x\in X.
$$
\item
We say that $f$ has the {\em $\mu$-shadowing property}
(for a given $\mu\in\mathcal{M}(X)$) if for every $\epsilon>0$ there are $\delta>0$ and $B\in\beta_X$ with $\mu(B)=1$ such that for every sequence $(x_n)_{n\geq0}$ with $x_0\in B$ and $d(f(x_n),x_{n+1})\leq\delta$ for all $n\geq0$ there is $x\in X$ such that
$d(f^n(x),x_n)\leq \epsilon$ for every $n\geq 0$.
\item
We say that $f$ has the {\em weak $\mu$-shadowing property}
(for a given $\mu\in\mathcal{M}(X)$) if for every $\epsilon>0$ there are $\delta>0$ and $B\in\beta_X$ with $\mu(X\setminus B)\leq \epsilon$ such that for every sequence $(x_n)_{n\geq0}$ with $x_0\in B$ and $d(f(x_n),x_{n+1})\leq\delta$ for all $n\geq0$ there is $x\in X$ such that
$d(f^n(x),x_n)\leq \epsilon$ for every $n\geq 0$.
\end{itemize}

Define
\begin{itemize}
\item
$
\mathcal{M}_{top*}(f)=\{\mu\in\mathcal{M}(X):f\mbox{ is $\mu$-topological set-valued stable}\},
$
\item
$
\mathcal{M}_{exp}(f)=\{\mu\in\mathcal{M}(X):f\mbox{ is $\mu$-expansive}\},
$
\item
$
\mathcal{M}_{sh}(f)=\{\mu\in\mathcal{M}(X):f\mbox{ has the $\mu$-shadowing property}\}
$
and
\item
$
{\small \mathcal{M}_{wsh}(f)=\{\mu\in\mathcal{M}(X):f\mbox{ has the weak $\mu$-shadowing property}\}}.
$
\end{itemize}

The sole difference between the $\mu$-shadowing and weak $\mu$-shadowing properties is given by the set $B$ which is full measure in the former definition and has measure at least $1-\epsilon$ in the latter.
Therefore, 
$$
\mathcal{M}_{sh}(f)\subset \mathcal{M}_{wsh}(f).
$$
But the converse inclusion is false (theorems 1 and 2 in \cite{lr}).
On the other hand, Theorem 6 in \cite{lr} (based on Theorem 3.1 in \cite{lm})
implies the following measurable version of Walters's stability theorem:
\begin{equation}
\label{litt}
\mathcal{M}_{exp}(f)\cap \mathcal{M}_{wsh}(f)\subset \mathcal{M}_{top*}(f).
\end{equation}

Now we introduce an alternative measure-theoretical version based on Walters's  concept of factor (Definition 2.6 in p. 61 of  \cite{w2}).

We say that $Y\subset X$ is {\em $g$-invariant} (for some $g\in C^0(X)$) if $g(Y)\subset Y$.
Given compact metric spaces $X, Z$, $\mu\in \mathcal{M}(X)$ and $\nu\in\mathcal{M}(Z)$, a measurable $g:X\to Z$ is {\em measure-preserving} if $g_*(\mu)=\nu$ where
 $g_*(\mu)=\mu\circ g^{-1}$ is the pull-back operator.

\begin{definition}
Given $\mu\in \mathcal{M}(X)$ we say that a measure preserving $f\in C^0(X)$ is a {\em factor} of a measure preserving
$g\in C^0(X)$ if there are $f$ (resp. $g$)-invariant sets $Y_f,Y_g\in \beta_X$ with
$$
\mu(Y_f)=\mu(Y_g)=1,
$$
and a measure preserving $h:Y_g\to Y_f$ such that
$$f\circ h=h\circ g.$$
\end{definition}

In this definition, $Y_f$ is assumed to be equipped with the
$\sigma$-algebra $\{Y_f\cap B:B\in\beta_X\}$, and the restriction of the measure $\mu$ to this $\sigma$-algebra (similarly for $Y_g$).


\subsection{Definition}
Our definition will be obtained
by weakening Walters's factor.
First, remove the measure preserving hypothesis since it is unnecessary.
Also, take $Y_f$ to be the whole $X$.
On the other hand, we will maintain the $g$-invariant set $Y=Y_g$ as it is fundamental for the identity $f\circ h=h\circ g$.
However, for a given $\epsilon>0$ and this measurable $h:X\to X$, Lusin Theorem permits to choose
a compact $Y\subset Y_g$
such that $\mu(Y_g\setminus Y)\leq \epsilon$
and $h:Y\to X$ is continuous. 
In particular,
$$
\mu(X\setminus Y)\leq \mu(X\setminus Y_g)+\mu(Y_g\setminus Y)
\leq \epsilon.
$$
As in topological stability, we use that $\epsilon$ to measure the $C^0$-distance between $h$ and the inclusion $i_Y$. 

Combining these properties with the definition of topological stability we arrive the following definition.

\begin{definition}
Given $\mu\in\mathcal{M}(X)$ we call $f\in C^0(X)$
{\em $\mu$-topologically stable} if for every $\epsilon>0$ there is $\delta>0$ such that for every $g\in C^0(X)$ with $d_{C^0}(f,g)\leq\delta$
 there are a compact $g$-invariant set $Y\subset X$ with
$\mu(X\setminus Y)\leq \epsilon$ and
$h:Y\to X$ continuous such that
$$
d_{C^0}(h,i_Y)\leq \epsilon\quad\mbox{ and }\quad
f\circ h=h\circ g.
$$
We denote
$$
\mathcal{M}_{top}(f)=\{\mu\in\mathcal{M}(X):f\mbox{ is $\mu$-topologically stable}\}.
$$
\end{definition}

A related example is given below.

\begin{example}
If $f\in C^0(X)$ is topologically stable, then
\begin{equation}
\label{octopussy}
\mathcal{M}_{top}(f)=\mathcal{M}(X).
\end{equation}
We don't know if, conversely, every $f\in C^0(X)$ satisfying \eqref{octopussy} is topologically stable.
\end{example}

\subsection{The result}
We state our main result.
Previously we introduce some basic definitions.
The first one is given in \cite{klm}.

Hereafter $X$ will denote a compact metric space.
The {\em orbit} of $p$ under $g\in C^0(X)$ is defined by
$$O_g(p)=\{g^n(p):n\geq0\}.$$
The closure operation is denoted by $\overline{B}$ for $B\subset X$.

\begin{definition}
We say that $p\in X$ is {\em topologically stable} for $f\in C^0(X)$ if
for every $\epsilon>0$ there is $\delta>0$ such that for every $g\in C^0(X)$ with $d_{C^0}(f,g)\leq\delta$ there is a continuous $h:\overline{O_g(p)}\to X$ such that
$$
d_{C^0}(h,i_{\overline{O_g(p)}})\leq\epsilon\quad\mbox{ and }\quad
f\circ h=h\circ g.
$$
Denote by $T(f)$ the set of topologically stable points of $f$.
\end{definition}

A Borel probability measure is {\em atomic} if there are points with positive measure
and {\em nonatomic} otherwise.
A classical example of atomic measure is the {\em Dirac measure} $m_p$ for $p\in X$
defined by $m_p(A)=0$ (if $p\notin A$) or $1$ (otherwise).
We say that $\mu\in\mathcal{M}(X)$ is {\em finitely supported} if it is a finite convex combination of Dirac measures i.e.
$$
\mu=\sum_{i=1}^kt_i m_{p_k},
$$
for some $t_1,\cdots, t_k\in (0,1)$ and
$\sum_{i=1}^kt_i=1$ and some $p_1,\cdots, p_k\in X$.
 We say that $\mu$
is {\em absolutely continuous} with respect to $\nu\in\mathcal{M}(X)$ (denoted $\mu\prec\nu$) if $B\in \beta_X$ and $\nu(B)=0$ imply $\mu(B)=0$.

A subset $\mathcal{M}\subset \mathcal{M}(X)$ is 
\begin{itemize}
\item
{\em $\prec$-invariant} if
$\nu\in\mathcal{M}$ and $\mu\prec\nu$ for $\mu\in\mathcal{M}$ $\Longrightarrow$ $\mu\in\mathcal{M}$;
\item
{\em convex} if $\mu,\nu\in\mathcal{M}$ and $0\leq t\leq 1$
$\Longrightarrow$ $t\mu+(1-t)\nu\in \mathcal{M}$. 
\end{itemize}

With these definitions we can state our main result.

\begin{theorem}
The following properties hold for all compact metric spaces $X$ and $f\in C^0(X)$:

\begin{enumerate}
\item
$T(f)=\{p\in X:m_p\in\mathcal{M}_{top}(f)\}$.

\medskip

\item
$\mathcal{M}_{top}(f)$ is $\prec$-invariant.

\medskip

\item
If $f$ is a homeomorphism and $X$ a closed manifold of dimension bigger than or equal $2$,
then $\mathcal{M}_{fs}(X)\cap \mathcal{M}_{top}(f)\subset \mathcal{M}_{sh}(f)$.

\medskip

\item
If $H:X\to Z$ is a homeomorphism of compact metric spaces, then
$\mathcal{M}_{top}(H\circ f\circ H^{-1})=H_*(\mathcal{M}_{top}(f))$.

\medskip

\item
If $f$ is expansive, then $\mathcal{M}_{top}(f)$ is convex.

\medskip

\item
$\mathcal{M}_{na}(X)\cap \mathcal{M}_{top}(X)\subset\mathcal{M}_{top*}(f)$.

\medskip

\item
If $f$ is expansive,
then
$
\mathcal{M}_{wsh}(f)\subset \mathcal{M}_{top}(f).
$ 
\end{enumerate}
\end{theorem}

Related to Item (6) above we have the following example.

\begin{example}
\label{basicas}
There are a compact metric space $X$ and $f\in C^0(X)$
such that $\mathcal{M}_{top}(f)\not\subset \mathcal{M}_{top*}(f)$.
\end{example}

\begin{proof}
Take any compact metric space $X$ with an isolated point $p$
and any $f\in C^0(X)$ such that $p$ is a fixed point of $f$.

Since $p$ is a fixed point of $f$ and also an isolated point of $X$, $p$ is a fixed point of every $g\in C^0(X)$ which is $d_{C^0}$-close to $f$.
Then, $p\in T(f)$ and so Item (1) of the theorem implies
$$
m_p\in\mathcal{M}_{top}(f).
$$
On the other hand, Theorem 7 in \cite{lr} (or the proof of Theorem 3.8 in \cite{lm}) implies
$$
m_p\notin \mathcal{M}_{top*}(f).
$$
\end{proof}

Let us apply this theorem to the case of homeomorphisms $f:S^1\to S^1$
where $S^1$ is the unit complex circle (circle homeomorphisms for short).
Denote by $P:\mathbb{R}\to S^1$ the projection
$P(t)=e^{2\pi t}$.
A lift of $f$ is a homeomorphism $F:\mathbb{R}\to \mathbb{R}$ such that
$f\circ P=P\circ F$.
It follows that
$$
\tau(f)=\lim_{n\to\infty}\frac{1}n(F^n(t)-t)
$$
is well-defined (mod $2\pi$).
This number is called {\em rotation number} of $f$.

\begin{example}
\label{parceria}
If $f$ is a circle homeomorphism with irrational rotation number, then
$\mathcal{M}_{top}(f)=\emptyset$. In particular, $\mathcal{M}_{top}(f)=\emptyset$ for all transitive circle homeomorphism $f$.
\end{example}

\begin{proof}
The second part of the example is a consequence of the first and Theorem 11.2.7 p. 397 in \cite{kh} (Poincar\'e Classification Theorem).
For the first part,
denote by $R_\alpha$ the circle rotation of angle $\alpha$.
It follows from the Poincar\'e Theorem that
that there is $k:S^1\to S^1$ continuous such that
$R_\alpha\circ k=k\circ f$ where $\alpha=\tau(f)$.

Suppose by contradiction that there is $\mu\in\mathcal{M}_{top}(f)$.
Choose $\delta>0$ from the $\mu$-topological stability of $f$ for
$\epsilon=\frac{1}2$.
Since every circle homeomorphism can be $C^0$-approximated by diffeomorphisms (c.f. Theorem 2.6 p. 49 in \cite{h}), by a theorem of Peixoto (Remark 3 p. 120 in \cite{p})
there is a Morse-Smale diffeomorphism $g:S^1\to S^1$
with $d_{C^0}(f,g)\leq\delta$.

For this $g$ there is a closed $g$-invariant set $Y\subset S^1$ with $\mu(S^1\setminus Y)\leq\frac{1}2$ and a continuous $h:Y\to S^1$ such that
$f\circ h=h\circ g$.
Since $g$ is Morse-Smale and $Y$ closed $g$-invariant (and nonempty because
$\mu(S^1\setminus Y)\leq\frac{1}2$), there is a periodic point $p\in Y$. Let $n$ be the period of $p$ i.e. $g^n(p)=p$.
By taking $z=k\circ h(p)$ we get
$$
R^n_\alpha(z)=R^n\circ k\circ h(p)=k\circ f^n\circ k(p)=k\circ h\circ g^n(p)=k\circ h(p)=z
$$
This contradicts that $\alpha$ is irrational completing the proof.
\end{proof}

Another example is as follows.

\begin{example}
$\mathcal{M}_{top}(R_\alpha)=\emptyset$ for every circle rotation $R_\alpha$.
\end{example}

\begin{proof}
By Example \ref{parceria} we can assume that $\alpha$ is rational.
Let $\delta$ be given by the $\mu$-topological stability of $R_\alpha$ for $\epsilon=\frac{diam(S^1)}4$.
Since $\alpha$ is rational, $R^k_\alpha=i_{S^1}$ for some $k\in \mathbb{N}$.
Choose an irrational $\beta$ such that
$d_{C^0}(R_\alpha,g)\leq\delta$ where $g=R_\beta$.
For this $g$ we choose $Y\subset S^1$ compact $g$-invariant with $\mu(S^1\setminus Y)\leq\epsilon$ and $h:Y\to S^1$ continuous such that
$d_{C^0}(h,i_Y)\leq\epsilon$ and $R_\alpha\circ h=h\circ g$.
We have that $g$ is minimal (for it is an irrational rotation).
Since $Y$ is closed invariant and obviously nonempty, we get $Y=S^1$.
For any $n\in\mathbb{N}$ and $x\in S^1=Y$ one has
\begin{equation}
\label{pasado}
h(x)=R^{nk}_\alpha(h(x))=h(g^{nk}(x))=h(R^{nk}_\beta(x))=h(R^n_{k\beta}(x)).
\end{equation}
Since $\beta$ is irrational, $k\beta$ also is hence $R_{k\beta}$ is minimal too.
Then, $O_{R_{k\beta}}(x)$ is dense in $S^1$. This together with \eqref{pasado}
implies $h$ is constant. Let then $x_\epsilon$ be such that $h(x)=x_\epsilon$ for every $x\in Y$.
It follows that
$$
d(x_\epsilon,x)=d(h(x),x)\leq d_{C^0}(h,i_Y)\leq\epsilon,\quad\quad\forall x\in Y=S^1.
$$
Then,
$diam(S^1)\leq 2\epsilon<diam(S^1)$ which is absurd.
This ends the proof.
\end{proof}

We express our gratitude to our colleagues at the SIMS Dynamical System Seminar in Sejong, South Korea, for encouraging and enlighten discussions related to this work.

\section{Proof of the theorem}

\noindent
We break it in a sequence of lemmas. The first one is Item (1) of the theorem.

\begin{lemma}
\label{depp}
For every compact metric space $X$ one has
$$
T(f)=\{p\in X:f\mbox{ is $m_p$-topologically stable}\},
\quad\quad\forall f\in C^0(X).
$$
\end{lemma}

\begin{proof}
Take $p\in T(f)$, $\epsilon>0$ and $\delta$ be given by the topological stability of $p$.
Fix $g\in C^0(X)$ with $d_{C^0}(f,g)\leq\delta$ and let $h:\overline{O_g(p)}\to X$
be given by the definition of of topologically stable point.
Letting $Y=\overline{O_g(p)}$ we get a closed $g$-invariant set
with $p\in Y$ hence  $m_p(X\setminus Y)=0\leq\epsilon$.
Moreover, $d_{C^0}(h,i_Y)\leq \epsilon$ and $f\circ h=h\circ g$,
therefore $f$ is $m_p$-topologically stable.

Conversely, suppose that
$f$ is $m_p$-topologically stable and choose $\epsilon>0$.
Let $\delta>0$ be given by the definition and $g\in C^0(X)$ with $d_{C^0}(f,g)\leq\delta$.
We can assume $0<\epsilon<1$ without loss pf generality.
Then, there are a closed $g$-invariant set $Y$ with $m_p(X\setminus Y)\leq\epsilon$,
$d_{C^0}(h,i_Y)\leq\epsilon$ and $f\circ h=h\circ g$.
In particular, $m_p(Y)\geq 1-\epsilon>0$ so $p\in Y$.
Since $Y$ is closed $g$-invariant,
$\overline{O_g(p)}\subset Y$ hence $h$ is defined on $\overline{O_g(p)}$.
Denote by $h_*=h|_{\overline{O_g(p)}}$ the restriction of $h$ to $\overline{O_g(p)}$.
Then, $d_{C^0}(h_*,i_{\overline{O_g(p)}})\leq d_{C^0}(h,i_Y)\leq\epsilon$
and also $f\circ h_*=h_*\circ g$ proving $p\in T(f)$.
\end{proof}

Now we prove item (2).

\begin{lemma}
\label{convidado}
Let $X$ be a compact metric space and $\mu,\nu\in \mathcal{M}(X)$.
If $f\in C^0(X)$ is $\nu$-topologically stable, then $f$ is $\mu$-topologically stable, $\forall \mu\in\mathcal{M}(X)$ with $\mu\prec\nu$.
\end{lemma}

\begin{proof}
Fix $\epsilon>0$ and $\mu\prec\nu$.
By the Radon-Nikodym Theorem we can choose $0<\epsilon'< \epsilon$ such that
$$
B\in \beta_X\quad\mbox{ and }\quad \nu(B)\leq\epsilon'\quad\Longrightarrow\quad \mu(B)\leq\epsilon.
$$
Let $\delta$ be given by the $\nu$-topological stability of $f$ for this $\epsilon'$.

Suppose that $g\in C^0(X)$ and $d_{C^0}(f,g)\leq\delta.$
Then, there is a compact $g$-invariant set $Y\subset X$
with $\nu(X\setminus Y)\leq\epsilon'$ and
a continuous $h:Y\to X$ such that
$d_{C^0}(h,i_Y)\leq\epsilon'$ and $f\circ h=h\circ g$.
It follows from the choice of $\epsilon'$ that
$d_{C^0}(h,i_Y)\leq\epsilon$ and also $\mu(X\setminus Y)\leq\epsilon$. This completes the proof.
\end{proof}

Given $\mu\in \mathcal{M}(X)$ we denote by $A(\mu)=\{p\in X:\mu(\{p\})>0\}$ the set of atoms of $\mu$. The following short lemma looks to be known we include the proof for completeness.

\begin{lemma}
\label{ross}
$A(\mu)=\{p\in X:m_p\prec \mu\}$.
\end{lemma}

\begin{proof}
Suppose $p\not\prec\mu$.
Then, $\mu(B)=0$ for some $B\in\beta_X$ with $m_p(B)>0$.
So, $p\in B$ thus $\mu(\{p\})\leq \mu(B)=0$ hence $p\notin A(\mu)$.
Therefore, $A(\mu)\subset \{p\in X:m_p\prec \mu\}$. 

Now suppose $m_p\prec \mu$ but $\mu(\{p\})=0$.
Then, $m_p(\{p\})=0$ which is impossible therefore $\{p\in X:m_p\prec \mu\}\subset A(\mu)$.
\end{proof}

Next we prove that the atoms of measures in $\mathcal{M}_{top}(f)$ are topologically stable points of $f$.

\begin{lemma}
\label{avo}
If $\mu\in \mathcal{M}_{top}(f)$, then $A(\mu)\subset T(f)$.
\end{lemma}

\begin{proof}
Take $p\in A(\mu)$. Then, $m_p\prec \mu$ by Lemma \ref{ross} so
$m_p\in\mathcal{M}_{top}(f)$ by Lemma \ref{convidado} thus
$p\in T(f)$ by Lemma \ref{depp}. 
\end{proof}

Now we prove Item (3).

\begin{lemma}
\label{louis}
Let $f:M\to M$ be a homeomorphism of a closed manifold of dimension $\geq2$, If
$\mu\in\mathcal{M}_{fs}(X)$ and $f$ is $\mu$-topologically stable, then $f$ has the $\mu$-shadowing property.
\end{lemma}

\begin{proof}
We will use the following definition from \cite{m}.
We say that $p\in M$ is {\em shadowable} if
for every $\epsilon>0$ there is $\delta>0$ such that
for every sequence $(x_i)_{i\geq0}$ with $x_0=p$ and $d(f(x_i),x_{i+1})\leq\delta$ for every $i\geq0$ there is $x\in M$ such that $d(f^i(x),x_i)\leq\epsilon$ for every $i\geq0$.
Denote by $Sh(f)$ the set of shadowable points of $f$.
The support $supp(\mu)$ of $\mu\in \mathcal{M}(M)$ is the set
of points $x$ such that $\mu(B[x,\delta])>0$ for every $\delta>0$.

Since $\mu$ is finitely supported, $supp(\mu)=A(\mu)$.
Then, $supp(\mu)\subset T(f)$ by Lemma \ref{avo}.
Since $M$ is a closed manifold of dimension $\geq2$,
$supp(\mu)\subset Sh(f)$ by Lemma 3.11 
so $f$ has the $\mu$-shadowing property by Lemma 2.5 in \cite{s}.
\end{proof}

Next we prove Item (4) i.e. the invariance of the $\mu$-topological stability under topological conjugacy.

\begin{lemma}
\label{invariance}
Let $X$ be a compact metric space and $\mu\in \mathcal{M}(X)$.
If $f\in C^0(X)$ is $\mu$-topologically stable, then $H\circ f\circ H^{-1}$ is $H_*(\mu)$-topologically stable for every homeomorphism of compact metric spaces $H:X\to Z$.
\end{lemma}

\begin{proof}
Fix $\epsilon>0$. Since $X$ and $Z$ are compact, $H$ is uniformly continous. Then, there is $\epsilon'>0$ such that
\begin{equation}
\label{enana}
a,b\in X\quad\mbox{ and }\quad d(a,b)\leq \epsilon'\quad\Longrightarrow\quad d(H(a),H(b))\leq\epsilon.
\end{equation}
For this $\epsilon'$ we choose $\delta'$ from the $\mu$-topological stability of $f$.
Again $H^{-1}$ is uniformly continuous so we can choose $\delta>0$ such that
\begin{equation}
\label{blanca}
c,d\in Y\quad\mbox{ and }\quad d(c,d)\leq \delta\quad\Longrightarrow\quad d(H^{-1}(c),H^{-1}(d))\leq\delta'.
\end{equation}
Now choose $g'\in C^0(Z)$ such that
$$
d_{C^0}(H\circ f\circ H^{-1},g')\leq\delta.
$$
Then, \eqref{blanca} implies
$$
d_{C^0}(f\circ H^{-1},H^{-1}\circ g')\leq\delta'\quad \mbox{ so }\quad d_{C^0}(f,H^{-1}\circ g'\circ H)\leq \delta'.
$$
It follows that there are a compact $(H^{-1}\circ g\circ H)$-invariant set $Y'\subset X$ and a continuous $h':Y'\to X$ such that
$$
d_{C^0}(h',i_{Y'})\leq \epsilon'\quad\mbox{ and }\quad f\circ h'=h'\circ H^{-1}\circ g\circ H.
$$
Then,
$H^{-1}\circ g'\circ H(Y')\subset Y'$ so $g'(H(Y'))\subset H(Y')$
thus $Y=H(Y')$ is $g'$-invariant
(and also compact for $H$ is a homeomorphism).

In addition,
$d(h'(y'),y')\leq \epsilon'$ for $y'\in Y'$ so
$$
d(h'(H^{-1}(z)),H^{-1}(z))\leq\epsilon'\quad\quad\forall z\in Y
$$
thus 
$$
d(H\circ h'\circ H^{-1}(z),z)\leq\epsilon\quad\quad\forall z\in Y
$$
hence \eqref{enana} implies
$$
d_{C^0}(h,i_Y)\leq \epsilon
$$
where
$$
h=H\circ h'\circ H^{-1}:Y\to Z.
$$
We also have
\begin{eqnarray*}
(H\circ f\circ H^{-1})\circ h & = & H\circ (f\circ h')\circ H^{-1}\\
& = & H\circ (h'\circ H^{-1}\circ g\circ H)\circ H^{-1}\\
& = & (H\circ h'\circ H^{-1})\circ g\\
& = & h\circ g
\end{eqnarray*}
namely
$$
(H\circ f\circ H^{-1})\circ h=h\circ g.
$$
Finally,
$$
H_*(\mu)(Z\setminus Y)=\mu(H^{-1}(Z\setminus Y))=\mu(X\setminus Y')\leq \epsilon
$$
proving the result.
\end{proof}

We now prove Item (5).

\begin{lemma}
\label{mama}
Let $f:X\to X$ be an expansive map of a compact metric space
and $\mu,\nu\in\mathcal{M}(X)$.
If $f$ is $\mu$-topologically and $\nu$-topologically stable, then $f$ is $(t\mu+(1-t)\nu)$-topologically stable for every $0\leq t\leq 1$.
\end{lemma}

\begin{proof}
Fix $\epsilon>0$ and let $e$ be an expansivity constant of $f$.
Let $\delta_\mu$ and $\delta_\nu$ be given by the $\mu$-topological and $\nu$-topological stabilities of $f$ respectively
for
$$
\epsilon'=\min\{\frac{e}2,\epsilon\}.
$$

Take $\delta=\min\{\delta_\mu,\delta_\nu\}$ and let $g\in C^0(X)$ be such that
$d_{C^0}(f,g)\leq\delta$.
Then, $d_{C^0}(f,g)\leq\delta_*$ for $*=\mu,\nu$ and so
there are compact $g$-invariant sets $Y_\mu,Y_\nu\subset X$ and continuous maps $h_\mu:Y_\mu\to X$ and $h_\nu:Y_\nu\to X$ such that
$$
d_{C^0}(h_*,i_{Y_*})\leq\epsilon'\quad\mbox{ and }\quad
f\circ h_*=h_*\circ g\quad\mbox{ for }\quad *=\mu,\nu.
$$
Define $Y=Y_\mu\cup Y_\nu$. Hence $Y$ is compact $g$-invariant and also
\begin{eqnarray*}
(t\mu+(1-t)\nu)(X\setminus Y) & = & t\mu(X\setminus Y)+(1-t)\nu(X\setminus Y)\\
& \leq & t\mu(X\setminus Y_\mu)+(1-t)\nu(X\setminus \nu)\\
& \leq & t\epsilon'+(1-t)\epsilon'\\
& = & \epsilon'\\
& \leq & \epsilon.
\end{eqnarray*}

We claim that $h_\mu=h_\nu$ on $Y_\mu\cap Y_\nu$.
Indeed, if $x\in Y_\mu\cap Y_\nu$,
then
\begin{eqnarray*}
d(f^n(h_\nu(x)),f^n(h_\nu(x))) & = & d(h_\nu(g^n(x)),h_\nu(g^n(x)))\\
& \leq & d(h_\mu(g^n(x)),g^n(x))+d(h_\nu(g^n),g^n(x))\\
& \leq & d_{C^0}(h,i_{Y_\mu})+d_{C^0}(h_\nu,i_{Y_\nu})\\
& \leq & 2\epsilon'\\
& \leq & e,\quad\quad\forall n\geq0.
\end{eqnarray*}
Since $e$ is an expansivity constant, we get $h_\mu(x)=h_\nu(x)$ proving the claim.
It follows that $h:Y\to X$ defined by
$$
h(x) = \left\{
\begin{array}{rll}
h_\mu(x),   & \hbox{if} & x \in Y_\mu \\
h_\nu(x),  & \hbox{if} & x \in Y_\nu
\end{array}
\right.
$$
is well-defined and continuous.
Since
$f\circ h_*=h_*\circ g$ for $*=\mu,\nu$,
we get $f\circ h=h\circ g$ proving the result.
\end{proof}

The lemma below is Item (6).

\begin{lemma}
\label{pai}
Let $f:X\to X$ be a continuous map of a compact metric space.
If $f$ is $\mu$-topologically stable for some nonatomic $\mu\in\mathcal{M}(X)$, then $f$ is a $\mu$-topologically set-valued stable measure.
\end{lemma}

\begin{proof}
Let $\epsilon>0$ and $\delta$ be given by the $\mu$-topological stability.
Take $g\in C^0(X)$ with $d_{C^0}(f,g)\leq\delta$.
Let $Y\subset X$ closed $g$-invariant with $\mu(X\setminus Y)\leq\epsilon$ and $h:Y\to X$ be given from the choice of $g$.
Define $H:X\to 2^X$ by
$$
H(x) = \left\{
\begin{array}{rll}
\emptyset,   & \hbox{if} & x \in Y \\
\{h(x)\},  & \hbox{if} & x \notin Y.
\end{array}
\right.
$$
Then, $Dom(H)=Y$ so
$$
\mu(X\setminus Dom(H))=\mu(X\setminus Y)\leq\epsilon.
$$
Since $d_{C^0}(h,i_Y)\leq\epsilon$, we also have
$d_{C^0}(H,i_X)\leq\epsilon$.
Since $\mu$ is nonatomic, $\mu\circ H=0$.
From $f\circ h=h\circ g$ we get $f\circ H=H\circ g$ completing the proof.
\end{proof}

The above lemma is false for general measures (by Example \ref{basicas}).
Finally we prove Item (7).

\begin{lemma}
\label{wally}
Let $X$ be a compact metric space and $\mu\in \mathcal{M}(X)$. Then, every expansive map with the weak $\mu$-shadowing property of $X$ is
$\mu$-topologically stable.
\end{lemma}

\begin{proof}
We will use the methods in \cite{klm}.

Let $f:X\to X$ be expansive with the weak $\mu$-shadowing property.
Fix $\mu\in\mathcal{M}(X)$.
Let $e>0$ be an expansivity constant.
Fix $\epsilon>0$ and $\epsilon'=\frac{1}8\min\{e,\epsilon\}$.
Let $\delta>0$ and $B\in \beta_X$ be given by the weak $\mu$-shadowing property for $\epsilon'$. 

Now, let $g:X\to X$ be continuous such that $d_{C^0}(f,g)\leq \delta$.
Take $Y=\overline{O_g(B)}$ the closure of the $g$-orbit of $B$ defined by
$$
O_g(B)=\{g^n(b):n\geq0,\,b\in B\}.
$$
Clearly $Y$ is closed $g$-invariant. Moreover,  $B\subset Y$ so
\begin{equation}
\label{chikatilo}
\mu(X\setminus Y)\leq \mu(X\setminus B)<\epsilon'<\epsilon.
\end{equation}

Take $x\in O_g(B)$ hence $x=g^{n_x}(b)$ for some $b\in B$ and some nonnegative integer $n_x$.
Since $d_{C^0}(f,g)\leq\delta$, the sequence $(b_n)_{n\geq0}$ defined by
$b_n=g^n(b)$ satisfy
$b_0=b\in B$ and
$$
d(f(b_n),b_{n+1})
=  d(f(g^n(b)),g(g^n(b)) \leq d_{C^0}(f,g)\leq\delta,\quad\quad\forall
n\geq0.
$$
Then, there is $y\in X$ such that
\begin{equation}
\label{frank}
d(f^n(y),g^n(b))\leq\epsilon',\quad\quad\forall n\geq0.
\end{equation}
Defining
$$
h(x)=f^{n_x}(y)
$$
we have a map $h:O_g(B)\to X$.

We shall prove that $h$ is well-defined.
Suppose $x=g^{m_x}(b')$ for some others nonnegative integer $m_x$ and $b'\in B$.
As before we get $y'\in X$ such that
$
d(f^n(y'),g^n(b'))\leq\epsilon'
$, $\forall n\geq0.$
Then,
\begin{eqnarray*}
d(f^r(f^{n_x}(y)),f^r(f^{m_x}(y'))) & \leq &
d(f^{r+n_x}(y),g^{r+n_x}(b))+\\
& & d(g^{r+m_x}(b'),f^{r+m_x}(y'))\\
& < & 2\epsilon'\\
& < & e,\quad\quad\forall r\geq0
\end{eqnarray*}
so $f^{n_x}(y)=f^{m_x}(y')$ by expansivity thus $h$ is well-defined.

Next we note that
if $x=g^{n_x}(b)$ as before, then $g(x)=g^{n_x+1}(b)$ and so
$h(g(x))=f^{n_x+1}(y')$ with $y'$ satisfying
$
d(f^n(y'),g^n(b))\leq \epsilon',
$ $\forall n\geq0.$
Combined with \eqref{frank} we get
$$
d(f^n(y),f^n(y'))\leq 2\epsilon'<e,\quad\quad\forall n\geq0
$$ so $y=y'$ by expansivity.
It follows that
$$
(f\circ h)(x)=f(f^{n_x}(y))=f^{n_x+1}(y)=h(g(x))=(h\circ g)(x),\quad\forall x\in O_g(B)
$$
proving
\begin{equation}
\label{wood}
f\circ h=h\circ g\quad\mbox{ on }\quad O_g(B).
\end{equation}
Also, by replacing $n$ by $n_x$ in \eqref{frank} we have
\begin{equation}
\label{under}
d(h(x),x)\leq \epsilon',\quad\quad\forall x\in O_g(B).
\end{equation}

We assert that $h:O_g(B)\to X$ is uniformly continuous.
Fix $\Delta>0$. Since $e$ is an expansivity constant we can apply Lemma 2 in \cite{w}
to find $N\in\mathbb{N}$ such that
$d(a,b)<\Delta$ whenever $a,b\in X$ satisfy
$d(f^n(a),f^n(b))\leq e$ for $0\leq n\leq N$.
On the other hand, $g$ is continuous so there is $\rho>0$ such that
$d(g^n(a),g^n(b))\leq \frac{e}2$ for $0\leq n\leq N$ whenever $a,b\in X$ and$d(a,b)<\rho$.
Now, take $x,y\in O_g(B)$ with $d(x,y)<\rho$.
It follows that
\begin{eqnarray*}
d(f^n(h(x)),f^n(h(y)))& = & d(h(g^n(x)),h(g^n(y)))\\
& \leq & d(h(g^n(x)),g^n(x))+d(g^n(x),g^n(y))+\\
& & d(h(g^n(y)),g^n(y))\\
& \overset{\eqref{under}}{\leq} & 2\epsilon'+\frac{e}2\\
& < & \frac{e}2+\frac{e}4\\
& < & e,\quad\quad\forall 0\leq n\leq N.
\end{eqnarray*}
Then, $d(h(x),h(y))<\Delta$ proving the assertion.

It follows that $h$ has a continuous extension $h:Y\to X$.
By taking limit in \eqref{wood} and \eqref{under} we obtain
$$
d_{C^0}(h,i_Y)\leq \epsilon'<\epsilon\quad\mbox{ and }\quad f\circ h=h\circ g.
$$
These expressions together with \eqref{chikatilo} imply the result.
\end{proof}

\begin{proof}[Proof of the theorem]
Items (1) and (2) were proved in lemmas \ref{depp} and \ref{convidado}
respectively. 
Items (3) to (7) in lemmas \ref{louis} to \ref{wally} respectively.
\end{proof}

\end{document}